\documentclass{birkjour}
\usepackage{graphicx}
\usepackage[utf8]{inputenc}   
\usepackage{fancyhdr}

\usepackage{color}

\usepackage{amsmath,amstext,amssymb,}
\usepackage{mathrsfs,amscd}

\usepackage[colorlinks]{hyperref}
\definecolor{wine-stain}{rgb}{0.5,0,0}
\hypersetup{
  colorlinks,
  linkcolor=wine-stain,
  filecolor=red,
  urlcolor=green,
  citecolor=blue
}

\newtheorem{thm}{Theorem}[section]
\newtheorem{cor}[thm]{Corollary}
\newtheorem{lem}[thm]{Lemma}
\newtheorem{prop}[thm]{Proposition}

\renewcommand{\proofname}{Proof}

\theoremstyle{definition}

\theoremstyle{definition}

\theoremstyle{definition}
\newtheorem{exam}[thm]{Example}

\def\R{\mathbb R}

\def\A{\mathcal A}
\def\B{\mathcal B}

\def\n{\mathbf n}

\def\A{\mathscr A}

\begin{document}

\title[]{Affine Focal Sets of Codimension $2$ \newline
 Submanifolds contained in Hypersurfaces}

\author[M.~Craizer]{Marcos Craizer}
\address{
Pontif\'icia Universidade Cat\'olica do Rio de Janeiro, Departamento de Matem\'atica, 22453-900 Rio de janeiro (RJ), BRAZIL}
\email{craizer@puc-rio.br}

\author[M.~J.~Saia]{Marcelo J.~Saia}
\address{Universidade de S\~ao Paulo, ICMC-SMA, Caixa Postal 668,
13560-970 S\~ao Carlos (SP), BRAZIL}
\email{mjsaia@icmc.usp.br}

\author[L.~F.~Sánchez]{Luis F.~Sánchez}
\address{Universidade Federal de Uberlândia, FAMAT, Departamento de Matemática,  Rua Goiás 2000, 38500-000 Monte Carmelo (MG),  BRAZIL}
\email{luis.sanchez@ufu.br}

\thanks{The first author wants to thank CNPq for financial support during the preparation of this manuscript. The second and third authors  have been partially supported by CAPES, FAPESP and CNPq.}



\date{August 08, 2016}

\begin{abstract}
In this paper we study the affine focal set, which is  the bifurcation set of the affine distance to submanifolds $N^n$ contained in hypersurfaces $M^{n+1}$ of the $(n+2)$-space. 
We give condition under which this affine focal set is a regular hypersurface and, for curves in $3$-space, we describe its stable singularities. For a given Darboux vector field
$\xi$ of the immersion $N\subset M$, one can define the affine metric $g$ and the affine normal plane bundle $\A$. We prove that the $g$-Laplacian of the position vector 
belongs to $\A$ if and only if $\xi$ is parallel. 

For umbilic and normally flat immersions, the affine focal set reduces to a single line. Submanifolds contained in hyperplanes or hyperquadrics are always normally flat. 
For $N$ contained in a hyperplane $L$, we show that $N\subset M$ is umbilic if and only if $N\subset L$ is an affine sphere and the envelope of tangent spaces is a cone. 
For $M$ hyperquadric, we prove that $N\subset M$ is umbilic if and only if $N$ is contained in a hyperplane. The main result of the paper is a general description of the umbilic
and normally flat immersions: Given a hypersurface $f$ and a point $O$ in the $(n+1)$-space, the immersion $(\nu,\nu\cdot(f-O))$, where $\nu$ is the co-normal of $f$, is umbilic 
and normally flat, and conversely, any umbilic and normally flat immersion is of this type. 
\end{abstract}

\maketitle

\section{Introduction}

Let $N^n\subset M^{n+1}\subset \R^{n+2}$ be a submanifold of codimension $2$ contained in a hypersurface $M$. For a frame $\{\xi,\eta\}$, $\xi$ tangent to $M$
and transversal to $N$, $\eta$ transversal to $M$, write
\begin{equation*}
D_XY=\nabla_XY +h^1(X,Y)\xi+h^2(x,y)\eta,
\end{equation*}
where $D$ denotes the canonical connection of the affine $(n+2)$-space, $X,Y$ are tangent to $N$ and $\nabla_XY$ is also tangent to $N$. We shall 
assume that the metric $h^2$ is non-degenerate, which is independent of the choice of the frame $\{\xi,\eta\}$. Under this non-degenerate hypothesis, 
there exists a unique, up to multiplication by a scalar function, vector field $\xi$ tangent to $M$ and transversal to $N$ such that $D_X\xi$ is tangent to $M$, for any $X$ 
tangent to $N$. Any such vector field is called a {\it Darboux vector field along $N$} (\cite{Craizer}). 

For a fixed Darboux vector field $\xi$, one can define an unique affine invariant metric $g$ on $N$ such that 
\begin{equation*}
\left[ X_1,...., X_n, D_{X_i}X_j,\xi \right]=\delta_{ij},
\end{equation*}
for any $g$-orthonormal frame $\{X_1,...,X_n\}$, where $[,,]$ denote the canonical volume form in the $(n+2)$-space (see \cite{LSan}). For $\eta$ transversal to $M$, write
\begin{equation*}
D_{X}\eta = -S_{\eta}X +\tau_2^1(X)\xi+\tau_2^2(X)\eta,
\end{equation*}
where $S_{\eta}X$ is tangent to $N$. It is proved in \cite{LSan} that there exists a vector field $\eta$ satisfying $\tau_2^2=0$ and 
\begin{equation*}
\left[ X_1,....,X_n, \eta,\xi \right]=1,
\end{equation*}
for any $g$-orthonormal frame $\{X_1,...,X_n\}$. Moreover, any other $\eta_1$ with the same properties can be written as $\eta_1=\eta+\lambda\xi$,
for some $\lambda\in\R$, and so the plane $\A(p)$, $p\in N$, generated by $\{\xi,\eta\}$ is unique. The family $\{\A(p),\ p\in N\}$ is called the 
{\it affine normal plane bundle}.

A vector field $\zeta\in\A$ is called {\it parallel} if $D_X\zeta$ is tangent to $N$, for any $X\in TN$. Some immersions admit a parallel Darboux vector field and, when this is the case, 
we have some reasons to choose it: The Transon planes coincide with the affine normal planes and the cubic form $C^2$ is apolar (\cite{Craizer}). In this paper we give one more reason
for the choice, namely, we prove that the $g$-Laplacian of the position vector $\phi$ belongs to the affine normal plane if and only if $\xi$ is parallel.

We are particularly interested in the case that the immersion admits a parallel Darboux vector field $\xi$ and also a parallel vector field $\eta\in\A$ linearly independent
of $\xi$. We verify that these conditions are equivalent to the flatness of a certain connection $\nabla^{\perp}$, called normal connection. When these conditions hold, 
we say that the immersion is {\it normally flat}. 

When $\xi$ is parallel and umbilic, i.e., $S_{\xi}$ is a multiple of the identity, then $N\subset M$ is a {\it visual contour}. This means that there exists a point $O\in\R^{n+2}$
in the intersection of all tangent spaces of $M$ along $N$. If $\eta\in\A$, linearly independent from $\xi$, is also parallel and umbilic, then there exists 
another fixed point $Q$ that belongs to the intersection of the normal planes $\A(p)$, $p\in N$. Thus if the immersion is umbilic and normally flat, there exist fixed points $O$ and $Q$
in the intersection of the tangent spaces and normal planes, respectively. Conversely, if such points $O$ and $Q$ exist, the immersion is necessarily umbilic and normally flat.

For $p\in N$, let $\{X_1(p),...,X_n(p)\}$ be a $g$-orthonormal tangent frame for $N$. The affine distance function $\Delta:\R^{n+2}\times N\to\R$  is given by
\begin{equation}\label{eq:DefineDelta}
\Delta(x,p)=\left[ x-p, X_1(p),..., X_n(p), \xi(p) \right].
\end{equation}
Then the singular set of $\Delta$ is the affine normal plane $\A$ (\cite{Sanchez}) and the bifurcation set $\B$ of $\Delta$ is called the {\it affine focal set}  of the immersion 
$N\subset M$. We verify that for umbilic and normally flat immersions, the set $\B$ reduces to a single line. We also give conditions under which this set is locally a hypersurface of 
$\R^{n+2}$. 

We consider in more detail curves $\phi$ contained in surfaces $M\subset \R^3$. In this case, the affine focal set $\B$ is a developable
surface $\B$ in $\R^3$ and we describe the stable singularities of $\B$, giving conditions under which they are equivalent to a cuspidal edge or a swallowtail. 
We verify that there exist parallel vector fields $\xi$ and $\eta$ generating the affine normal plane bundle and so the immersion is normally flat. For curves, the umbilic condition must be 
replaced by the visual contour condition together with the existence of a fixed point $Q$ in the intersection of the affine normal planes.
We characterize the curves satisfying these conditions and apply the result for spatial curves not necessarily contained in surfaces. 
We also discuss two six vertex theorems, one affinely invariant for "umbilic" spatial curves and the other projectively invariant for general planar curves.

For immersions $N\subset M$ such that $N$ is contained in a hyperplane $L$, we can choose $\xi$ parallel and $\eta$ the Blaschke vector field of $N\subset L$, which is also parallel. 
In this context we verify that $N\subset M$ is an umbilic immersion if and only if $N\subset L$ is an affine sphere and the envelope of tangent spaces of $N\subset M$
is a cone. 

For immersions $N\subset M$, where $M$ is a hyperquadric, we can choose a parallel Darboux vector field $\xi$ $h$-unitary and orthogonal to $N$, where $h$ is the Blaschke metric of $M$. 
Then the vector field $\eta(p)=p-Q$, where $Q$ is the center of $M$, is umbilic, parallel and contained in the affine normal plane. In this context, we show that the immersion $N\subset M$
is umbilic if and only if $N$ is contained in a hyperplane.

The main result of this paper is a geometric characterization of umbilic and normally flat immersions. Consider a non-degenerate immersion $f:U\subset\R^{n}\to\R^{n+1}$ and let 
$\nu:U\subset\R^{n}\to\R^{n+1}$ denote its Blaschke co-normal map. We verify that the centro-affine  immersion $\left(\nu,\nu\cdot(f-O)\right):U\to\R^{n+2}$ is umbilic and normally flat, 
for any choice of the origin $O\in\R^{n+1}$. The converse is also true: Any umbilic and normally flat immersion is locally obtained by this construction.  As a corollary we show that a umbilic 
and normally flat immersion $N$ is contained in a hyperplane if and only if $f$ is a proper affine sphere.

The paper is organized as follows: In section 2, we describe the Darboux vector fields and other basic facts of the affine geometry of codimension $2$ submanifolds contained 
in a hypersurface. In section 3 we consider the case of parallel Darboux vector field and prove that the Laplacian of the position vector is in the affine normal plane. 
In section 4 we discuss conditions under which an immersion is normally flat and umbilic. In section 5 we describe the basic properties of the affine focal set, 
giving conditions for its regularity. In section 6 we consider curves contained in surfaces of $\R^3$,
where the affine focal set coincides with the envelope of normal planes. In this case, we give a complete classification of the stable singularities that appear in the affine focal set. In section 7
we study the particular cases of submanifolds contained in a hyperplane, when the immersion is umbilic if and only if the envelope of tangent spaces is a cone over an affine sphere, and
of submanifolds contained in hyperquadrics, where the immersion is umbilic if and only if the submanifolds is contained in a hyperplane. In section 8 we prove the main result of the paper: 
Given a hypersurface $f$ and a point $O$ in the $(n+1)$-space, the immersion $(\nu,\nu\cdot(f-O))$, where $\nu$ is the co-normal of $f$, is umbilic and normally flat, and, conversely, any umbilic and normally flat immersion is of this type.

\section{Affine geometry of codimension 2 submanifolds contained in hypersurfaces}

Denote by $D$ the canonical affine connection and by $\left[\cdot,...,\cdot\right]$ the canonical volume form of $\R^{n+2}$.

\subsection{Basic equations for codimension $2$ immersions}

Let $\phi:U\to\R^{n+2}$ be an immersion, where $U$ is a $n$-dimensional manifold and denote $N=\phi(U)$. 
Consider a transversal plane bundle $\A(p)$, $p\in N$, and let $\{\xi_1,\xi_2\}$ be a frame for $\A$. For $X,Y\in \mathfrak{X}(U)$, write
\begin{equation}\label{eq:basic1}
D_X\phi_*Y=\phi_*(\nabla_XY)+h^1(X,Y)\xi_1+h^2(X,Y)\xi_2.
\end{equation}
Then $\nabla$ is a torsion free affine connection and $h^i$, $i=1,2$,  are symmetric bilinear forms on $U$. For $X\in\mathfrak{X}(U)$, $i=1,2$, write
\begin{equation}\label{eq:basic2}
D_X\xi_i=-\phi_*(S_i(X))+\tau_i^1(X)\xi_1+\tau_i^2(X)\xi_2,
\end{equation}
where $\tau_i^j$ are $1$-forms on $U$ and the linear maps $S_i(u):T_uU\to T_uU$ are called {\it shape operators} of $\xi_i$. In general, for $\zeta\in\A$, write
\begin{equation*}\label{eq:basic3}
D_X\zeta=-\phi_*(S_{\zeta}(X))+\tau_{\zeta}^1(X)\xi_1+\tau_{\zeta}^2(X)\xi_2,
\end{equation*}
and the linear map $S_{\zeta}$ is called the shape operator of $\zeta$. Most of the time we shall consider that $U=N$ and $\phi:N\to\R^{n+2}$ is the inclusion map.

The normal connection
$\nabla^{\perp}$ is defined by 
\begin{equation*}\label{eq:basic4}
\nabla^{\perp}_X\zeta=\tau_{\zeta}^1(X)\xi_1+\tau_{\zeta}^2(X)\xi_2.
\end{equation*}
The curvature tensor of the normal connection, called \textit{ normal curvature tensor},
$R_{\nabla^{\perp}}:T_p N\times T_pN\times\A_p\rightarrow\A_p$, is defined by
\begin{equation*}
R_{\nabla^{\perp}}(X,Y)\zeta=\nabla^{\perp}_X (\nabla^{\perp}_Y \zeta)-\nabla^{\perp}_Y (\nabla^{\perp}_X \zeta)-\nabla^{\perp}_{[X,Y]}\zeta.
\end{equation*}
Since $\R^4$ has vanishing curvature, we obtain
\begin{equation}\label{eq:Codazzi}
R_{\nabla^{\bot}}(X,Y)\zeta=\left( h^1(X, S_{\zeta}Y)- h^1(Y, S_{\zeta}X) \right)\xi_1+  \left( h^2(X, S_{\zeta}Y)- h^2(Y, S_{\zeta}X) \right)\xi_2.
\end{equation}
We have also Ricci equations, $i=1,2$, $X,Y\in\mathfrak{X}(N)$,
{\scriptsize
\begin{equation}\label{eq:basic6a}
h^1(X,S_1Y)-h^1(Y,S_1X)=d\tau_1^1(X,Y)+\tau_1^2(Y)\tau_2^1(X)-\tau_2^1(Y)\tau_1^2(X),
\end{equation}
\begin{equation}\label{eq:basic6b}
h^2(X,S_1Y)-h^2(Y,S_1X)=d\tau_1^2(X,Y)+\tau_1^1(Y)\tau_1^2(X)-\tau_1^2(Y)\tau_1^1(X)+\tau_1^2(Y)\tau_2^2(X)-\tau_2^2(Y)\tau_1^2(X),
\end{equation}
\begin{equation}\label{eq:basic6c}
h^1(X,S_2Y)-h^1(Y,S_2X)=d\tau_2^1(X,Y)+\tau_2^2(Y)\tau_2^1(X)-\tau_2^1(Y)\tau_2^2(X)+\tau_2^1(Y)\tau_1^1(X)-\tau_1^1(Y)\tau_2^1(X),
\end{equation}
\begin{equation}\label{eq:basic6d}
h^2(X,S_2(Y))-h^2(Y,S_2(X))=d\tau_2^2(X,Y)+\tau_2^1(Y)\tau_1^2(X)-\tau_1^2(Y)\tau_2^1(X).
\end{equation}
}
The cubic form $C^i$, $i=1,2$, is defined by
\begin{equation}\label{eq:CubicForm}
C^i(X,Y,Z)=\nabla_Xh^i(Y,Z)+\tau_1^i(X)h^1(Y,Z)+\tau_2^i(X)h^2(Y,Z),
\end{equation}
for $X,Y,Z\in\mathfrak{X}(N)$. It follows from Codazzi equations that $C^i$ is symmetric in $X,Y,Z$. 

For more details, see \cite{Nomizu-Vrancken} and \cite{Sanchez}.

\subsection{Darboux vector field}

From now on, we consider a codimension $2$ submanifold $N$ contained in a hypersurface $M$ of $\R^{n+2}$. Take vector fields $\xi_1=\xi$ tangent to $M$ and transversal to $N$ and 
$\xi_2=\eta$ transversal to $M$. We shall assume that $h^2$ given by equation \eqref{eq:basic1} is non-degenerate, which is independent of the choice of $\xi$ and $\eta$. 

Under this non-degeneracy hypothesis, there exists a unique $\xi$, up to multiplication by a scalar function, such that $D_X\xi$ is tangent to $M$, for any $X\in TN$. 
We call any such $\xi$ a {\it Darboux vector field} of the immersion $N\subset M$ (see \cite{Craizer}). 
Equation \eqref{eq:basic1} can be re-written as 
\begin{equation}\label{eq:basic}
D_XY=\nabla_XY+h^1(X,Y)\xi+h^2(X,Y)\eta.
\end{equation}
Now equation \eqref{eq:basic2} implies that $\tau_1^2=0$ and so, by equation \eqref{eq:basic6b}, $S_1$ is $h^2$-self adjoint.  

\subsection{Affine metric and affine normal plane bundle}

Consider an immersion $N^n\subset M^{n+1}\subset\R^{n+2}$ and denote by $\xi$ a fixed Darboux vector field tangent to $M$ along $N$. Consider a tangent frame 
$\mathfrak{u}=\{X_1,...,X_n\}$ and let
\begin{equation*}
G_{\mathfrak{u}}(X,Y)=\left[ X_1,....,X_n,D_XY,\xi \right]. 
\end{equation*}
Then $G_{\mathfrak{u}}$ is non-degenerate and we can define the  metric.
\begin{equation*}
g(X,Y)=\frac{G_{\mathfrak{u}}(X,Y)}{|\det_{\mathfrak{u}}G_{\mathfrak{u}}|^{\frac{1}{n+2}}},
\end{equation*}
where $\det_{\mathfrak{u}}G_{\mathfrak{u}}= \det (G_{\mathfrak{u}}(X_i,X_j))$. The metric $g$ is called the {\it affine metric} of the immersion
$N\subset M$ (\cite{LSan},\cite{LSanTese}).

Consider a tangent $g$-orthonormal frame $\{X_1,...,X_n\}$ and let $\eta$ be a transversal vector field satisfying 
\begin{equation*}\label{eq:Det1}
\left[ X_1,...,X_n,\eta,\xi \right]=1.
\end{equation*}
Then we obtain from equation \eqref{eq:basic} that $g=h^2$. 
It turns out that there exists a transversal vector field $\eta$ such that $D_X\eta$ is tangent to $M$, for any $X$ tangent to $N$. Moreover, the plane bundle generated by $\{\xi,\eta\}$ 
is unique, and it is called the {\it affine normal plane bundle} (\cite{LSan},\cite{LSanTese}). We shall denote it by $\A(p)$, $p\in N$. 

From equation \eqref{eq:basic2}, $\tau_2^2=0$, which, by equation \eqref{eq:basic6d}, implies that $S_2$ is $h^2$-self adjoint. 
We conclude that, in this setting, both $S_1$ and $S_2$ are $g$-self adjoint.

\subsection{Affine semiumbilic immersions}

An immersion $N\subset M$ is {\it affine semiumbilic} with respect to $\A$ if $S_{\zeta}$ is a multiple of the identity, for some $\zeta\in\A$. An 
immersion $N\subset M$ is affine semiumbilic with respect to $\A$ at a point $p$ if $S_{\zeta}(p)$ is a multiple of the identity, for some $\zeta\in\A$.

\begin{prop}
If an immersion $N\subset M$ is affine semiumbilic with respect to $\A$ at $p$, then the shape operators $S_{\zeta}(p)$, $\zeta\in \A$, commute. For $n=2$, the converse also holds.
\end{prop}

\begin{proof}
Fix $\zeta_1,\zeta_2\in\A$. If $p$ is affine semi-umbilic, we can write $aS_{\zeta_1}+bS_{\zeta_2}=\lambda I$, for some real numbers $a,b,\lambda$. 
This implies that $S_{\zeta_1}\circ S_{\zeta_2}=S_{\zeta_2}\circ S_{\zeta_1}$. 
For the converse, write for $n=2$
\begin{equation*}\label{eq:S2X2}
\left\{
\begin{array}{c}
S_{1}(X_1)=\lambda_{11}X_1+\lambda_{21}X_2\\
S_{1}(X_2)=\lambda_{12}X_1+\lambda_{22}X_2
\end{array}
\right.
\ \ \left\{
\begin{array}{c}
S_{2}(X_1)=\mu_{11}X_1+\mu_{21}X_2\\
S_{2}(X_2)=\mu_{12}X_1+\mu_{22}X_2,
\end{array}
\right.
\end{equation*}
where $\{X_1,X_2\}$ is a $g$-orthonormal frame of $TN$. Since both $S_{1}$ and $S_{2}$ are $g$-self adjoint, $\lambda_{12}=\lambda_{21}$ and $\mu_{12}=\mu_{21}$.
The semiumbilic condition 
is then equivalent to 
$$
\det\left[
\begin{array}{cc}
\lambda_{22}-\lambda_{11} & \lambda_{12} \cr
\mu_{22}-\mu_{11} & \mu_{12}
\end{array}
\right]=0.
$$
On the other hand, the commutativity of $S_{1}$ and $S_{2}$ is also equivalent to this condition. 
\end{proof}

\section{Parallel Darboux vector field}

For immersions $N\subset M$ that admit a parallel Darboux vector field $\xi$, this choice 
of $\xi$ is ubiquitous. In fact, it is proved in \cite{Craizer} that the condition $\xi$ parallel is equivalent to the coincidence of the affine normal plane with the Transon plane, and
also equivalent to the apolarity of the second cubic form. In this section, we show that $\xi$ parallel is also equivalent
to the condition that the Laplacian of the position vector $\phi$ belongs to the affine normal plane.

\subsection{Immersions that admit a parallel Darboux vector field}

It is not always true that an immersion $N\subset M$ admits a parallel Darboux vector field $\xi$. In fact, we have the following proposition:

\begin{prop}
The immersion $N\subset M$ admits a local parallel Darboux vector field if and only $R_{\nabla^{\perp}}\xi=0$. When it exists, the parallel 
Darboux vector field is unique up to a multiplicative constant.
\end{prop}
\begin{proof}
If the Darboux vector field $\xi$ is parallel, $\tau_1^1=0$, and so equations \eqref{eq:Codazzi} and \eqref{eq:basic6a}
imply that $R_{\nabla^{\perp}}\xi=0$. Conversely, assuming $R_{\nabla^{\perp}}\xi=0$, the same equations imply that $\tau_1^1$ is locally exact.
Then $\exp(\lambda)\xi$ is parallel, where $d\lambda=-\tau_1^1$. Finally, if $\xi$ is parallel, $a\xi$ is parallel if and only if $a$ is constant.
For more details see \cite{Craizer}.
\end{proof}

\subsection{The apolarity condition}

Since $\tau_1^2=\tau_2^2=0$, equation \eqref{eq:CubicForm} says that  
\begin{equation*}
C_2(X,Y,Z)=(\nabla_Xh^2)(Y,Z),
\end{equation*}
for $X,Y,Z\in\mathfrak{X}(N)$. Equivalently, denoting $K(X,Y)=\nabla_XY-\hat{\nabla}_XY$, we have
\begin{equation*}
C_2(X,Y,Z)=-2h^2(K(X,Y),Z).
\end{equation*}
The cubic form $C_2$ is {\it apolar} with respect to $h^2$ if 
\begin{equation*}
tr_{h^2}C_2(X,\cdot,\cdot)=0,
\end{equation*}
for any $X\in \mathfrak{X}(N)$. This condition is equivalent to $tr_{h^2}K=0$. In \cite{Craizer}, the following proposition is proved:

\begin{prop}
The cubic form $C_2$ is apolar with respect to $h^2$ if and only if $\xi$ is parallel. 
\end{prop}

\subsection{The Laplacian operator}

Denote by $\Delta$ the Laplacian operator with respect to the metric $g$. Recall that $\Delta(f)$ is the trace with respect to $g$
of the Hessian operator defined by
$$
Hess(f)(X,Y)=XY(f)-\hat\nabla_XY(f),
$$
where $\hat\nabla$ denotes the Levi-Civita connection of $g$ (see \cite{Nomizu}, p.64).

\begin{prop}
The Laplacian of $\phi$ belongs to the affine normal plane if and only if $\xi$ is parallel. In this case
\begin{equation}\label{eq:Lap}
\frac{1}{n}\Delta\phi=\eta-\lambda\xi,
\end{equation}
where $\lambda=-\frac{1}{n}tr_g(h^1)$.
\end{prop}
\begin{proof}
Write 
\begin{equation*}
D_X\phi_*Y-\phi_*(\hat\nabla_XY)=\phi_{*}(K(X,Y))+h^1(X,Y)\xi+h^2(X,Y)\eta.
\end{equation*}
where $K(X,Y)=\nabla_XY-\hat\nabla_XY$. For $\xi$ parallel, the cubic form $C^2$ is apolar with respect to $h^2$, i.e., $tr_g(K)=0$. Thus 
$$
\Delta\phi=tr_g(h^1)\xi+n\eta,
$$
thus proving the proposition. Conversely, if $\Delta(\phi)$ belongs to the affine normal plane, $tr_g(K)=0$, which implies that $\xi$ is parallel.
\end{proof}

\subsection{Visual contours}

A submanifold $N\subset M$ is a {\it visual contour} if there exists $O\in\R^{n+2}$ such that the tangent space to $M$ at each point of $N$ passes through $O$. 
This class of submanifolds is the object of study of the {\it centro-affine differential geometry} and is important in computer graphics (\cite{Cipolla}).

For a visual contour, one can choose the Darboux vector field $\xi=\phi-O$ which is both parallel and also umbilic, i.e., $S_{\xi}$ is a multiple of the identity, for any $p\in N$. To prove the converse, 
we need the following lemma:

\begin{lem}\label{lem:ConstantCurvature}
Assume $n\geq 2$.  If the Darboux vector field $\xi$ is both umbilic and parallel, then $S_{1}=\alpha I$, for some constant $\alpha$. 
\end{lem}
\begin{proof}
For $p\in N$, write $S_{1}(X)=\alpha(p)X$. From  equation 
$$
D_XD_Y\xi-D_YD_X\xi-D_{[X,Y]}\xi=0,
$$
we obtain
$$
X(\alpha)Y-Y(\alpha)X=0.
$$
Taking $X,Y$ linearly independent, we conclude that $X(\alpha)=Y(\alpha)=0$, which implies that $\alpha$ is constant. 
\end{proof}

\begin{cor}\label{cor:VisualContour}
Assume $n\geq 2$. If there exists a Darboux vector field $\xi$ which is both umbilic and parallel, then $N\subset M$ is a visual contour. 
\end{cor}
\begin{proof}
Take $\xi$ umbilic and parallel. By the lemma \ref{lem:ConstantCurvature}, $S_{1}=\sigma I$, for some constant $\sigma$. Writing  $O=p+\sigma^{-1}\xi$, we get $X(O)=0$ and so $O$ is constant, which implies that $N\subset M$ is a visual contour.
\end{proof}

\section{Normally flat immersions}\label{sec:ParallelVectorFields}

An immersion $N\subset M$ is {\it normally flat} with respect to the affine normal plane bundle $\A$ if $R_{\nabla^{\perp}}=0$. 
An immersion $N\subset M$ is {\it (totally) umbilic} if $S_{\zeta}$ is a multiple of the identity, 
for any $\zeta\in\A$, for any $p\in N$. In this section we study the conditions under which the immersions are normally flat and umbilic.

\subsection{Conditions for normally flat immersions}

\begin{prop}
Fix a parallel Darboux vector field $\xi$. Then there exists a parallel vector field $\eta\in\A$ linearly independent with $\xi$ if and only if $\nabla^{\perp}$ is flat. 
When the parallel vector field $\eta$ exists, any other parallel vector field in $\A$ is necessarily of the form $\eta+a\xi$, for some constant $a$.
\end{prop}
\begin{proof}
If $\eta\in\A$ is parallel, $\tau_2^1=0$, and so equations \eqref{eq:Codazzi} and \eqref{eq:basic6c}
imply that $R_{\nabla^{\perp}}\eta=0$. Thus $\nabla^{\perp}$ is flat. Conversely, if $R_{\nabla^{\perp}}\eta=0$, 
the same equations imply that $\tau_2^1$ is locally exact. Then $\eta+\lambda\xi$ is parallel, where $d\lambda=-\tau_2^1$. Finally, if $\eta$ is parallel,
$\eta+a\xi$ is parallel if and only if $a$ is constant.
\end{proof}

\begin{cor}
An immersion $N\subset M$ is normally flat if and only if it admits a parallel Darboux vector field $\xi$ and a parallel vector 
field $\eta\in\A$, linearly independent from $\xi$. The vector field $\xi$ is unique up to a multiplicative constant, while $\eta$ is unique up to the addition
of $a\xi$, $a$ constant. 
\end{cor}

\subsection{Conditions for umbilic and normally flat immersions}

\begin{prop}\label{prop:UNF}
Assume that $n\geq 2$. If $N\subset M$ is umbilic and normally flat, then it is a visual contour and there exists a point $Q\neq O$ contained in the intersection of the affine normal planes $\A(p)$, $p\in N$. 
Conversely, if $N\subset M$ is a visual contour and the intersection of the normal planes $\A(p)$, $p\in N$, contains a point $Q\neq O$, then the immersion is umbilic and normally flat. 
\end{prop}

\begin{proof}
Assuming $N\subset M$ is umbilic and normally flat, we can choose $\xi$ umbilic and parallel. Then corollary \ref{cor:VisualContour} implies that $N\subset M$ is a visual contour, 
and so we can assume $\xi=\phi-O$. 
We can also choose $\eta$ umbilic and parallel. So we may write $S_{\eta}=\mu I$, and, by Lemma \ref{lem:ConstantCurvature}, $\mu$ is constant.
Writing $Q=\phi+\mu^{-1}\eta$, we have $Q\in \A(p)$ and $D_XQ=0$, for any $X\in TN$, which implies that $Q$ is independent of $p\in N$. 

Conversely, if $N\subset M$ is a visual contour and $Q\in\A(p)$, we may assume $\xi=\phi-O$ and define $\eta=\phi-Q$. Then $\eta$ is umbilic and parallel, which implies that
$N\subset M$ is umbilic and normally flat.
\end{proof}

\section{Affine focal sets}

In this section we define the affine distance and show that its singular set coincides with the affine normal plane. We give also conditions 
for the regularity of the affine focal set, which is the bifurcation set of the affine distance. Finally, by comparing with the envelope of tangent spaces, we show that
all simple singularities are realizable.

\subsection{Affine distance and its singular set}

The affine distance $\Delta:\R^{n+2}\times U\to\R$ is defined by 
\begin{equation}\label{eq:DefineDelta1}
\Delta(x,u)=\left[ x-\phi(u), X_1(u),..., X_n(u), \xi(u) \right].
\end{equation}

\begin{lem}
The singular set of $\Delta$ is $\{\A(p)| p\in N\}$.
\end{lem}

\begin{proof}
Differentiating equation \eqref{eq:DefineDelta1} with respect to $u_1$ we obtain
\begin{equation*}
\frac{\partial\Delta}{\partial u_1}=\left(\sum_{j=1}^n \Gamma_{1j}^j +\tau_1^1(X_1) \right)\Delta+\left[ x-\phi, \eta, X_2, ...,X_n, \xi \right]=0.
\end{equation*}
Differentiating $[X_1,...,X_n,\eta,\xi]=1$ with respect to $u_1$ we obtain
$$
\sum_{j=1}^n\Gamma_{1j}^j+\tau_1^1(X_1)=0. 
$$
We conclude that
\begin{equation*}\label{eq:DeriveDeltaU}
\frac{\partial\Delta}{\partial u_1}=\left[ x-\phi, \eta, X_2,....X_n, \xi \right].
\end{equation*}
Similarly, differentiating equation \eqref{eq:DefineDelta} with respect to $u_k$ we obtain
\begin{equation}\label{eq:DeriveDeltaV}
\frac{\partial\Delta}{\partial u_k}=\left[ x-\phi, X_1,.., \eta, ...,X_n, \xi \right].
\end{equation}
We conclude that the singular set of $\Delta$ is defined by  $x-\phi=a\xi+b\eta$, $a,b\in\R$, and the lemma is proved.
\end{proof}

\subsection{The bifurcation set of $\Delta$}

Take a $g$-orthonormal basis $\{X_1,...,X_n\}$ formed by eigenvectors of $S_2$, i.e., $S_2X_j=\mu_jX_j$. 
Write also
\begin{equation*}
S_{1} (X_k)=\sum_{j=1}^n \sigma_{kj} X_j.
\end{equation*}

\begin{prop}\label{prop:DegreeN}
Write $x=\phi+a\xi+b\eta$. The bifurcation set $\mathcal{B}_p$ of $\Delta$ at $p$ is given by $q(a,b)=0$, where $q$ is a polynomial of degree at most $n$ in $(a,b)$. 
\end{prop}
\begin{proof}
Differentiating equations \eqref{eq:DeriveDeltaV} we obtain 
$$
\frac{\partial^2\Delta}{\partial u_k^2}=1-b\mu_{k}-a\sigma_{kk},\ \ \frac{\partial^2\Delta}{\partial u_k\partial u_l}=-a\sigma_{kl}.
$$
Taking $q(a,b)=\det(D^2\Delta)$, the lemma is proved.
\end{proof}

\begin{cor}
For an umbilic and normally flat immersion, the affine focal set reduces to a line.
\end{cor}
\begin{proof}
For umbilic and normally flat immersions, $\mu_k=\mu$ and  $\sigma_{kj}=\sigma\delta_{kj}$.
\end{proof}

\begin{cor}
If the shape operators commute at $p$, the affine focal set at this point consists of $n$ lines. In particular this holds if $p$ is affine semiumbilic.
\end{cor}

\begin{proof}
If the shape operators commute at $p$, then we can find an orthonormal frame $\{X_1,...,X_n\}$ made of $\zeta$-principal directions, for all $\zeta\in\A$. In this basis, $\sigma_{kl}=0$, for $k\neq l$, 
and we write $\sigma_{kk}=\sigma_k$. From  we conclude that the bifurcation set is given by
$$
\prod_{k=1}^n(1-b\mu_{k}-a\sigma_{k})=0,
$$
which is equivalent to $1-b\mu_{k}-a\sigma_{k}=0$, for some $1\leq k\leq n$.
\end{proof}

\subsection{Conditions for regularity of the affine focal set}

Assume that $\mu_1$ is a simple eigenvalue of $S_{2}$ associated with the principal direction $X_1$. We shall give conditions under which 
the affine focal set is smooth at $x=\phi+\mu_1^{-1}\eta$. 

\begin{lem}\label{lem:Regular1}
Consider $a=0$ and $b=\mu_1^{-1}$ in proposition \ref{prop:DegreeN}. At this point, the vector $\zeta=\mu_1\xi-\sigma_{11}\eta$ is tangent to the curve $q(a,b)=0$.
\end{lem}
\begin{proof}
Differentiating $q(a,b)=0$ and taking $a=0$, $b=\mu_1^{-1}$ we obtain 
$$
(q_a,q_b)=-\prod_{j=2}^{n}(1-\mu_j\mu_1^{-1})  \left( \sigma_{11}, \mu_1\right),
$$
and the lemma is proved.
\end{proof}

\begin{lem}\label{lem:Regular2}
Write
\begin{equation}
x(u)=\phi(u)+\mu_1^{-1}(u)\eta(u), \ \ t\in U.
\end{equation}
Assuming that $\eta$ is parallel, 
$$
x_{u_1}=-\frac{X_1(\mu_1)}{\mu_1^2}\eta,
$$
and for $j\neq 1$, 
$$
x_{u_j}=(1-\mu_1^{-1}\mu_j)X_j-\frac{X_j(\mu_1)}{\mu_1^2}\eta.
$$
\end{lem}

Denote by $E_1$ the eigenspace generated by the eigenvector $X_1$ of $S_{2}$, and by $E_1^{\perp}\subset T_pN$ the $(n-1)$-space $g$-orthogonal to $E_1$.

\begin{prop}
Assume that $\eta$ is parallel. If $\mu_1$ is a simple eigenvalue of $S_{2}$ and $X_1(\mu_1)\neq 0$, the affine focal set is smooth at $x=\phi+\mu_1^{-1}\eta$ 
and the tangent space of $\mathcal{B}$ at this point is $E_1^{\perp}\oplus A(p)$.
\end{prop}
\begin{proof}
By lemma \ref{lem:Regular1}, $\zeta$ is contained in the $(n+1)$-dimensional subspace $E_1^{\perp}\oplus A(p)$. By lemma \ref{lem:Regular2}, $x_{u_i}$, $i=1,...n$, are also
contained in this subspace. Thus we must only prove that these vectors are linearly independent. But this follows from the hypothesis $X_1(\mu_1)\neq 0$ and $\mu_j\neq\mu_1$, if $j\neq 1$. 
\end{proof}

\subsection{Relation with the Envelope of Tangent Spaces and simple singularities}

The Envelope of Tangent Spaces is the set
$$
ET_N=\{ x\in\R^{n+2}|\ \Delta=\Delta_{u_1}=...=\Delta_{u_n}=0, \ \mathrm{for\ some}\ u\in U\},
$$
or equivalently,
$$
ET_N=\{x=p+u\xi(p)|\ p\in N, u\in\R\}.
$$
If $u\neq \sigma^{-1}$, for some non-zero eigenvalue $\sigma$ of $S_1$, then $ET_N$ is regular (for details, see \cite{Craizer}). 

If $u=\sigma^{-1}$, for some non-zero eigenvalue $\sigma$ of $S_1$, then $x=p+u\xi(p)$ is a point of intersection of $ET_N$ with $\mathcal{B}$. In \cite{Craizer},
one can find examples of such points where $\Delta(x,p)$ are versal deformations of the simple singularities, namely, $A_k$, $k\geq 2$, $D_k$, $k\geq 4$, $D_6$, $D_7$ and $D_8$
(for examples in case $n=2$, see section \ref{sec:Hyperplanar} and \cite{Davis}, ch.8).
These examples are also examples of singularities that appear in $\mathcal{B}$. We conclude that any simple singularity appears as a singular point of the
affine focal set $\mathcal{B}$ for some immersion $N\subset M$ .

\subsection{An example}

\begin{exam}Let $\alpha$  and  $\beta$ be a plane curves parametrized by arc length affine and we  consider the surface $N$ parameterized by 
 \[
 \phi(u_1,u_2)=(\alpha(u_1),\beta(u_2)).
 \]
Denote $X_1=\phi_{u_1},\ X_2=\phi_{u_2}$ and let $\xi=(\alpha''(u_1), \beta''(u_2))$. We shall assume that $\xi$ is tangent to a hypersurface $M\subset\R^4$. Since
$\xi'$ is tangent to $N$, $\xi$ is the parallel Darboux vector fields of the immersion $N\subset M$. 
Let $\xi_1=(\alpha''(u_1),0)$, $\xi_2=(0,\beta''(u_2))$. Then $\xi_1$ and $\xi_2$ are parallel and belongs to the affine normal plane bundle. Moreover
$$
\left[ X_1,X_2,\xi_1,\xi_2 \right]=1.
$$

The affine distance function  $F:\R^2 \times \R^4 \rightarrow \R$ on product of curves is given by:
\[F(u,x)=[X_1,X_2,\xi,x-\phi(u_1,u_2)]\]

The bifurcation set is given by

\[\mathcal{B}_F=\{x \in \R^4: F_{u_1}=F_{u_2}=F_{u_1u_1} F_{u_2u_2}-F_{u_1u_2}^2=0\}\]

\begin{align*}
 F_{u}&=-[x-\phi(u_1,u_2),X_2, \xi_1,\xi_2]\\
 F_{v}&=-[X_1, x-\phi(u_1,u_2), \xi_1,\xi_2]
\end{align*}

We have $F_{u_1}=F_{u_2}=0$ if, and only if, there are $r,s \in \R$ such that $x-\phi(u_1,u_2)=r \xi+ s \eta$. Now for $x\in \R^4$ such that $F_{u_1}=F_{u_2}=0$

\begin{align*}
 F_{u_1u_1}&=1-r k(\alpha)\\
 F_{u_1u_2}&=0\\
F_{u_2u_2}&=-1+sk(\beta)
\end{align*}
where $k$ denote the affine curvature of the plane curves $\alpha$ and $\beta$. Therefore $x \in \mathcal{B}_F$, if and only if,  $x-\phi(u_1,u_2)=r \xi_1+ s \xi_2$ and $r=k(\alpha)^{-1}$
or $s=k(\beta)^{-1}$. Thus, at each point, the affine focal set is a pair of lines concurring at $(r,s)=\left(k(\alpha)^{-1},k(\beta)^{-1}\right)$.

Globally, we can write the affine focal set as $E(\alpha)\times\R^2\cup \R^2\times E(\beta)$, where $E$ denotes the affine planar evolute of the planar curve. 
\end{exam}

\section{Curves contained in surfaces of $\R^3$}

In this section, we shall consider a curve $\phi:U\to M\subset\R^3$, where $U\subset\R$ is an interval and $M$ is a surface of $\R^3$. 

\subsection{Affine normal plane bundle}

Let $\phi:U\to M\subset\R^3$ be a curve. The non-degeneracy hypothesis says that, at each point, the osculating plane of $\phi$ does not
coincide with the tangent plane of $M$. Under this hypothesis, there exists a reparameterization of $\phi$ such that $\phi'''(u)$ is tangent to $M$.
From now on, we shall assume that this property holds.

Choose a vector field $\xi(u)$ in the Darboux direction satisfying 
\begin{equation}\label{eq:Det2}
[T,\phi'',\xi]=1,
\end{equation}
where $T=\phi'(u)$. 
\begin{prop}
The vector field $\xi$ is parallel, $g=du$ is the affine metric and the affine normal plane bundle
is spanned by $\{\xi,\phi''\}$. 
\end{prop}
\begin{proof}
Differentiating equation \eqref{eq:Det2} we obtain that $\xi'(u)\parallel T$, and thus $\xi$ is parallel. Moreover, equation \eqref{eq:Det2} 
implies that $g(T)=1$, i.e., $g=du$. Finally equation \eqref{eq:Det2} together with $\gamma'''$ tangent to $M$ implies that $\gamma''(u)$ belongs to the affine normal plane.
\end{proof}

We can write 
\begin{equation}\label{eq:StructCurves}
\xi'=-\sigma T;\ \ \ \phi'''=-\rho T+\tau \xi.
\end{equation}

\subsection{Normally flat immersion}

Take a real function $\lambda$ such that $\lambda'=-\tau$. Note that $\lambda$ may not be globally defined. 
Consider the vector field $\eta=\phi''+\lambda\xi$ (see Figure \ref{fig:ParallelVF}). Then $[T,\eta,\xi]=1$ and 
$$
\eta'=-\rho T+\tau\xi-\tau\xi-\lambda\sigma T=-(\rho+\lambda\sigma) T.
$$
Defining $\mu=\rho+\lambda\sigma$, we obtain the following equations:
$$
\left\{
\begin{array}{c}
T'=\eta-\lambda\xi\\
\eta'=-\mu T\\
\xi'=-\sigma T
\end{array}
\right.
$$
Note that $\xi$ and $\eta$ are both parallel and thus the immersion is normally flat with respect to the affine normal plane bundle $A$.  

\begin{figure}[htb]
 \centering
\includegraphics[width=0.50\linewidth]{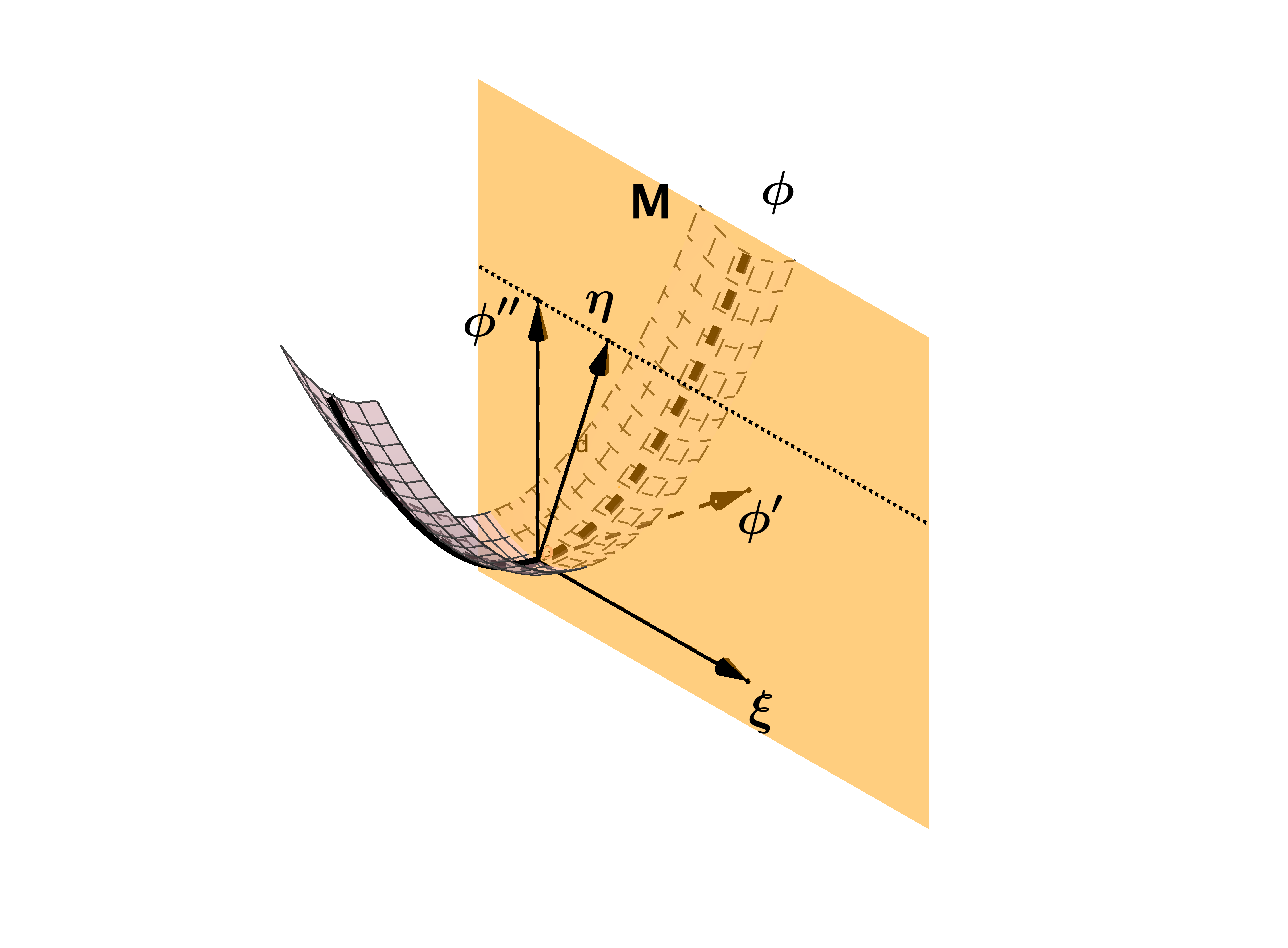}
\caption{ The parallel vector fields $\xi$ and $\eta$. }
\label{fig:ParallelVF}
\end{figure}

\begin{lem}
A curve $\phi:U\to M$ is a visual contour if and only if $\sigma$ is constant.
\end{lem}
\begin{proof}
If $\sigma$ is constant, then $O=\phi+\sigma^{-1}\xi$ is constant and belong to the tangent plane to $M$ at $p$, thus implying that
$\phi$ is a visual contour. Conversely, if $O$ belongs to the tangent plane to $M$ at each $p$, then $\phi-O$ is a parallel Darboux vector 
field. Thus $\xi=a(\phi-O)$, for some constant $a$, and hence $\sigma$ is constant.
\end{proof}

\begin{prop}\label{prop:SigmaMuConstant}
A curve $\phi:U\to M$ is a visual contour and there exists a constant point $Q\in\A(p)$, for each $p\in N$,  if and only if $\sigma$ and $\mu$ are constant.
\end{prop}
\begin{proof}
If $\mu$ is constant, then $Q=\phi+\mu^{-1}\eta$ is constant and belongs to $\A(p)$, for each $p$. Conversely, if there exists a point $Q\in\A(p)$, for each $p$, 
then $\phi-Q$ is parallel and belongs to $\A(p)$. Thus $\eta=\phi-Q+a\xi$, for some constant $a$, which implies that $\mu=-1+a\sigma$ is constant.
\end{proof}

Comparing the above proposition with proposition \ref{prop:UNF} we see that the umbilic and normally flat condition for $n\geq 2$ corresponds 
to $\sigma$ and $\mu$ constant for curves. 

\subsection{Envelope of affine normal planes}

The affine distance is given by 
\begin{equation}
\Delta(x,u)=\left[ x-\phi(u), \phi'(u), \xi(u) \right].
\end{equation}
Thus 
\begin{equation*}
\Delta_u(x,u)=\left[ x-\phi(u), \phi''(u), \xi(u) \right],
\end{equation*}
and so $\Delta_u=0$ is the equation of the affine normal plane. 
We conclude that the affine focal set 
\begin{equation*}
\mathcal{B}=\{x\in\R^3|\ \Delta_u=\Delta_{uu}=0, \ \mathrm{for\ some} \ u\in U\},
\end{equation*}
coincides with the envelope of $\A(u)$, $u\in U$. 

Denote $Q(u)=\phi(u)+\mu^{-1}(u)\eta(u)$ and  $O(u)=\phi(u)+\sigma^{-1}(u)\xi(u)$ and let $l(u)$ be the line connecting the points $O(u)$ and $Q(u)$ (see Figure \ref{fig:Linel}).

\begin{prop}
The envelope $\B$ of the family of planes $\{\A(u),u\in U\}$ is the ruled surface formed by the lines $l(u)$, $u\in U$.
\end{prop}
\begin{proof}
The equation of the plane $\A(u)$ is $F(x,u)=0$, where
$$
F(x,u)=[x-\phi(u),\eta(u),\xi(u)].
$$
Differentiating this equation with respect to $u$ we obtain
\begin{equation}\label{eq:Ft}
F_u=-\mu [x-\phi, T,\xi]-\sigma [X-\phi,\eta,T]-1.
\end{equation}
Writing $x-\phi=a\xi+b\eta$ and substituting in this equation we get $F_u=0$ if and only if
\begin{equation*}\label{eq:CurvesB}
a\sigma+b\mu=1,
\end{equation*}
thus proving the proposition.
\end{proof}

\begin{figure}[htb]
 \centering
\includegraphics[width=0.50\linewidth]{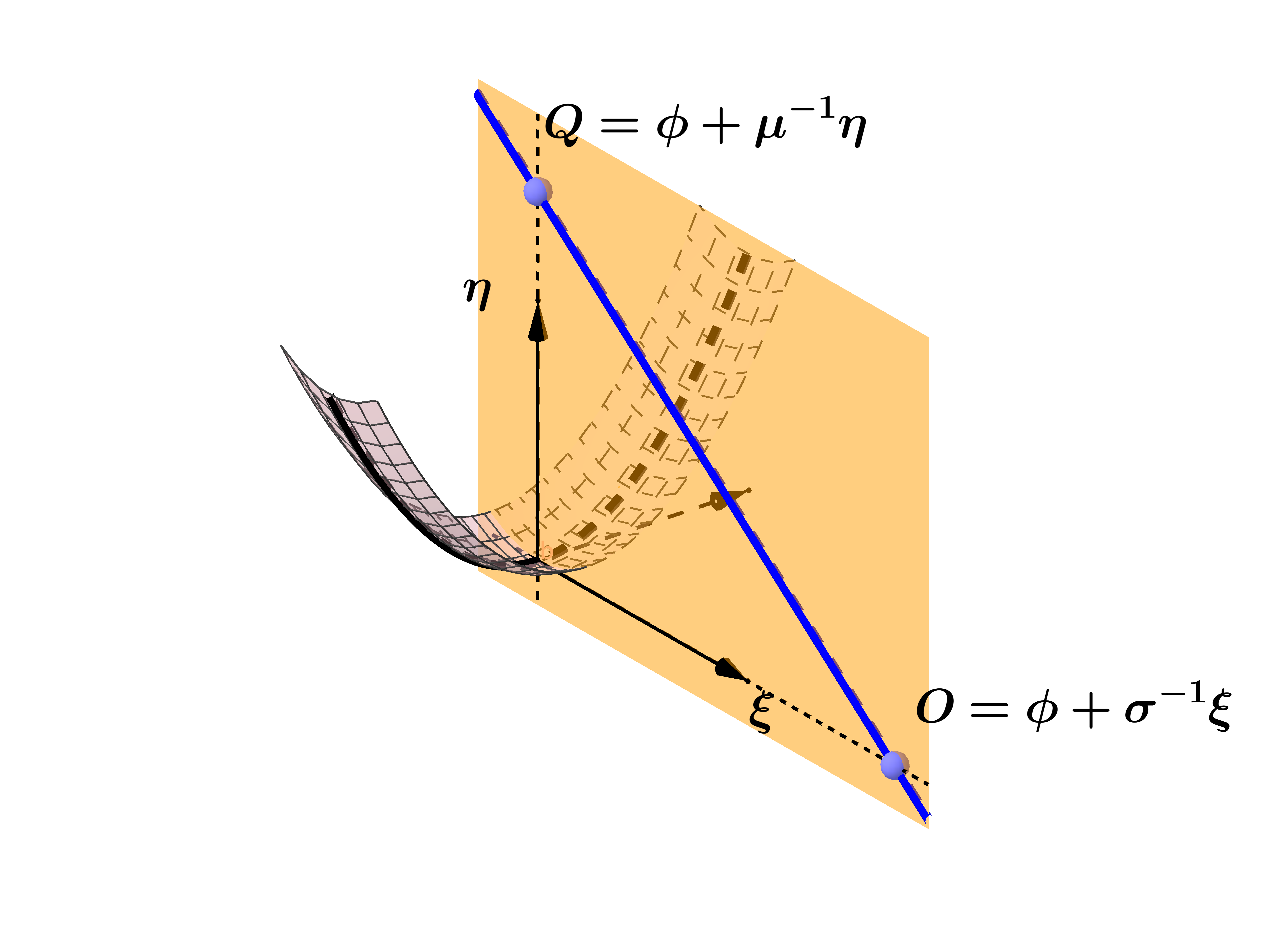}
 \caption{ The line $l(u)$ whose union is the affine focal set $\mathcal{B}$ . }
\label{fig:Linel}
\end{figure}

\begin{lem}
The ruled surface $\B$ is developable. 
\end{lem}
\begin{proof}
Since
$$
Q'=-\frac{\mu'}{\mu^2}\eta,\ \ O'=-\frac{\sigma'}{\sigma^2}\xi,
$$
we conclude that
$$
[Q',O',O-Q]=0,
$$
thus proving that $\B$ is developable. 
\end{proof}

\subsection{Singularities of $\B$}

\begin{thm}
Let ${\mathcal B}$ be the affine focal set of $N\subset M$. Each point of $\mathcal{B}$ at $p\in N$ belongs to the line
$$
a\sigma+b\mu=1.
$$
Then
\begin{enumerate}
\item ${\mathcal B}$ is smooth if $a\sigma'+b\mu'\neq 0$.
\item ${\mathcal B}$ is locally diffeomorphic to a cuspidal edge if $a\sigma'+b\mu'=0$ and $a\sigma''+b\mu''\neq 0$.
\item ${\mathcal B}$ is locally diffeomorphic to a swallowtail if $a\sigma''+b\mu''=0$ and $a\sigma'''+b\mu'''\neq 0$.
\end{enumerate}
\end{thm}

\begin{proof}
Differentiating equation \eqref{eq:Ft} we obtain
\begin{equation}\label{eq:Ftt}
F_{uu}=-\mu'\left[x-\phi,T,\xi\right]-\sigma'\left[x-\phi,\eta,T\right]+(\lambda\sigma-\mu)F
\end{equation}
Thus $F_{uu}=0$ if and only if $a\sigma'+b\mu'=0$.
Differentiating equation \eqref{eq:Ftt} and discarding the last parcel we obtain
\begin{equation}\label{eq:Fttt}
F_{uuu}=-\mu''\left[x-\phi,T,\xi\right]-\sigma''\left[x-\phi,\eta,T\right]+(\lambda\sigma'-\mu')F
\end{equation}
Thus $F=F_u=F_{uu}=F_{uuu}=0$ if and only if $a\sigma''+b\mu''=0$. Differentiating equation \eqref{eq:Fttt} and substituting the values of $(a,b)$ we obtain
\begin{equation*}\label{eq:Ftttt}
F_{uuuu}=a\sigma'''+b\mu'''.
\end{equation*}
The result follows now from \cite{Giblin}, ch.6.
\end{proof}

\subsection{Curves with $\mu$ and $\sigma$ constant}

In this section we describe those immersions $\phi\subset M$ for which $\sigma$ and $\mu$ are constant. For these curves,
the affine focal set ${\mathcal B}$ reduces to a single line. 
We may assume, w.l.o.g., that $\sigma=-1$ and $\xi=\phi$. Thus equation \eqref{eq:StructCurves} can be written as
\begin{equation}\label{eq:CurvesCA}
\phi'''=-\rho\phi'+\tau\phi.
\end{equation}
The condition $\mu'=0$ can be written as 
\begin{equation}\label{eq:RolinhaTau}
\rho'+\tau=0,
\end{equation}
and thus this equation reduces to
\begin{equation*}
\phi'''=-\left( \rho\phi \right)',
\end{equation*}
or equivalently,
\begin{equation*}
\phi''=-\rho\phi +Q,
\end{equation*}
for some constant vector $Q$. Assuming $Q=(0,0,1)$ and writing $\phi(u)=(\gamma(u),z(u))$, for some planar curve $\gamma(u)$, we obtain the discoupled equations
\begin{equation}\label{eq:ODE}
\gamma''(u)=-\rho(u)\gamma(u),\ \ \ z''(u)=-\rho(u) z(u)+1.
\end{equation}

\bigskip

We say that a planar curve $\Gamma(u)$ is parameterized by affine arc-length if $\left[ \Gamma',\Gamma'' \right]=1$. In this case 
$$
\Gamma'''(u)=-\rho(u)\Gamma'(u),
$$
for a certain scalar function called the {\it affine curvature} of $\Gamma$. The function 
$$
z(u)=\left[  \Gamma(u)-O,\Gamma'(u)  \right]
$$ 
is called the {\it affine distance}, or {\it support function}, of $\Gamma$ with respect to the origin $O$ (\cite{Cecil},\cite{Nomizu}). 
Next theorem will be generalized to higher dimensions in section \ref{sec:Umbilic}:

\begin{thm}\label{thm:UmbilicCurves}
Given a planar curve $\Gamma$ parameterized by affine arc-length with affine curvature $\rho$, let $\gamma=\Gamma'$ and $z$ denote the support function of $\Gamma$ with respect to an arbitrary origin $O$. Then $(\gamma,z)$ satisfies equation \eqref{eq:ODE} and, conversely, any solution of \eqref{eq:ODE} is obtained in this way.
\end{thm}
\begin{proof}
Assume that $\gamma=\Gamma'$ and $z$ is the support function of $\gamma$ with respect to $O$. Then 
$$
\gamma''(u)=\Gamma'''(u)=-\rho(u)\Gamma'(u)=-\rho(u)\gamma(u).
$$
Moreover $z'(u)=[\Gamma(u)-O,\Gamma''(u)]$ and so
$$
z''(u)=1+\left[\Gamma(u)-O,\Gamma'''(u)\right] =1-\rho(u)z(u),
$$
thus proving that $(\gamma,z)$ satisfies equation \eqref{eq:ODE}. Since the general solution of the linear system of ODE's \eqref{eq:ODE} is $2$-dimensional, we conclude
that $(\gamma,z)$ is a general solution of this system.
\end{proof}

A flattening point of $\phi$ is a point $\phi(u_0)$ such that $\tau(u_0)=0$.

\begin{cor}
Consider a closed curve $\phi$ with constant $\sigma$ and $\mu$. Then $\phi$ has at least six flattening points.
\end{cor}

\begin{proof}
If $\phi$ is closed, then $\Gamma$ is a closed convex planar curve. By the affine six vertex theorem (\cite{Buchin},\cite{Fidal}), there exists at least six points 
of $\Gamma$ where $\rho'(u_0)=0$. This implies that $\tau(u_0)=0$ at these points. 
\end{proof}

\subsection{Curves in $\R^3$}

Theorem \ref{thm:UmbilicCurves} is also interesting in the context of spatial curves, not contained in a surface. Assume $\Phi:U\to\R^3$ 
is a smooth spatial curve parameterized by affine arc-length, i.e., $\left[ \Phi'(u),\Phi''(u),\Phi'''(u)\right]=1$. Then we can write 
\begin{equation}\label{eq:SpatialCurves}
\Phi''''=-\rho\Phi''+\tau\Phi',
\end{equation}
for some scalar functions $\rho$ and $\tau$. The plane passing through $\Phi(u)$ and generated by $\Phi'(u)$ and $\Phi'''(u)$ is called {\it affine
rectifying plane}, while the envelope $RS(\Phi)$ of these planes is called the {\it intrinsic affine binormal developable} of $\Phi$. It is proved in  \cite{Izu1}
that $RS(\Phi)$ is a cylindrical surface if and only if $\rho'+\tau=0$. 

Writing $\phi=\Phi'$ in equation \eqref{eq:SpatialCurves}, we obtain equation \eqref{eq:CurvesCA}. Theorem \ref{thm:UmbilicCurves} says that $\mu$ is constant
if and only if $\phi=(\gamma, z)$, where $\gamma=\Gamma'$, for some curve $\Gamma$ and $z=\left[\Gamma-O,\gamma\right]$. Thus we conclude that
$\Phi=(\Gamma,Z)$, where
\begin{equation}\label{eq:Z} 
Z(u)=\int_{u_0}^u \left[ \Gamma(v)-O,\Gamma'(v)  \right] dv.
\end{equation}
The function $Z(u)$ is the area of the planar region bounded by the angle $\Gamma(u_0)O\Gamma(u)$ and the arc of $\Gamma$, $u_0\leq v\leq u$. 

\begin{cor}
Consider a spatial curve $\Phi:U\to\R^3$ parameterized by affine arc-length, and let $\rho$ and $\tau$ be given by equation \eqref{eq:SpatialCurves}. 
Then the intrinsic affine binormal developable $RS(\Phi)$ is cylindrical if and only if we can write $\Phi=(\Gamma,Z)$, where $\Gamma$ is a planar curve and $Z$ is given by equation \eqref{eq:Z}, for some origin $O$ and some initial point $u_0$. 
\end{cor}
\begin{proof}
We have that $RS(\Phi)$ is cylindrical if and only if $\rho'+\tau=0$. By equation \eqref{eq:RolinhaTau}, this condition is equivalent to $\mu$ constant. 
Then Theorem \ref{thm:UmbilicCurves} says that $\mu$ constant is equivalent to $\phi=(\gamma,z)$, where $\gamma=\Gamma'$ for some planar curve $\Gamma$ and $z$ the support function of $\Gamma$ with respect to some origin $O$. We conclude that $RS(\Phi)$ cylindrical is equivalent to $\Phi=(\Gamma,Z)$, where $Z$ is given by equation \eqref{eq:Z}.
\end{proof}

\subsection{A projectively invariant six vertex theorem}

The quantity $\rho'+2\tau$ is projectively invariant. In fact, $(\rho'(u)+2\tau(u))^{1/3}du$ is the {\it projective length} of the curve $\phi$ and $\rho'+2\tau=0$ 
if and only if $\phi$ is contained in a quadratic cone (for details, see \cite{Guieu}).

The curve $\phi$ is projectively equivalent to a planar curve and in this case $\tau=0$. Thus
$$
\phi'''=-\rho\phi'
$$
and so $\rho$ is the affine curvature of $\phi$. By the six vertex theorem for planar closed convex curves (\cite{Buchin}), $\phi$ admits at least six points 
where $\rho'=0$, or equivalently, $\rho'+2\tau=0$. 
This means that a closed convex curve $\phi$ admits at least six points with higher order contact with a quadratic cone.

\section{Two classes of normally flat immersions}

In this section we consider two classes of normally flat submanifolds $N\subset M$: (1) $N$ contained in a hyperplane $L$ and (2) $M$ hyperquadric. 

\subsection{Submanifolds contained in hyperplanes}\label{sec:Hyperplanar}

For $N$ contained in a hyperplane $L$, take $\xi$ tangent to $M$ such that $\xi\cdot\n=1$, where $\cdot$ denotes inner product and $\n$ is the euclidean normal of $L$. Then $\xi$ 
is the parallel Darboux vector field of the immersion $N\subset M$. 

Let $\eta$ and $h$ denote the Blaschke normal vector field and metric of $N\subset L$, respectively. Then $\eta$ is parallel and, for a $h$-orthonormal frame of $TN$ we have
$$
\left[ X_1,..., X_n, \eta \right]=1.
$$
Since $g=h$, $\{X_1,...,X_n\}$ is $g$-orthonormal and 
$$
\left[ X_1,..., X_n, \eta, \xi \right]=1.
$$
Thus $\eta\in A$. We conclude that $\B\cap L$ is exactly the affine focal set of the immersion $N\subset L$ (\cite{Davis}, ch.8).

\begin{prop}
Let $N\subset M$ be an immersion such that $N$ is contained in a hyperplane $L$. Then $N\subset M$ is umbilic if and only if
$N\subset L$ is an affine sphere and the envelope of tangent spaces of $N\subset M$ is a cone.
\end{prop}
\begin{proof}
The vector field $\eta$ is umbilic if and only if $N\subset L$ is an affine sphere. The vector field $\xi$ is umbilic if and only the envelope of tangent spaces is a cone. 
\end{proof}

\subsection{Submanifolds contained in hyperquadrics}

In this section we consider the case $M$ hyperquadric. By an affine transformation of $\R^{n+2}$, we may assume that $M$ is given by
$$
\sum_{i=1}^{n+2}\epsilon_i x_i^2=1,
$$
where $\epsilon_i=\pm 1$. Denote by $<\cdot,\cdot>$ the non-degenerate, possibly indefinite, metric  
$$
<A,B>=\sum_{i=1}^{n+2} \epsilon_i A_iB_i.
$$
Then the tangent space of $M$ is $\phi^{\perp}$, where the orthogonality is taken with respect to $<\cdot,\cdot>$. Moreover, one can verify that
$\phi$ is the affine Blaschke normal and $<\cdot,\cdot>$ restricted to the tangent space of $M$ is the affine Blaschke metric $h$ of $M$.

Consider now $N\subset M$. The non-degeneracy hypothesis implies that $<v,v>\neq 0$, for any $v\in T_pN$. This implies that we can choose $\xi$ orthogonal to $N$ with 
\begin{equation}\label{eq:XiUnitary}
<\xi,\xi>=\epsilon, \ \ \epsilon=\pm 1.
\end{equation}

\begin{lem} 
We have that:
\begin{enumerate}
\item $\xi$ is a parallel vector field in the Darboux direction. 
\item The metric $g$ is the restriction of $h$ to $N$. 
\end{enumerate}
\end{lem}
\begin{proof}
To prove 1, write
$$
D_X\xi=\bar\nabla_X\xi+h(X,\xi)\eta,
$$
where $\bar\nabla_X\xi$ is tangent to $M$. Since $h(X,\xi)=0$, we conclude that $D_X\xi$ is tangent to $M$ and so $\xi$ is a Darboux vector field of $N\subset M$. 
Differentiating equation \eqref{eq:XiUnitary} 
we obtain  $<D_X\xi,\xi>=0$, for any $X$ tangent to $N$, which implies that $D_X\xi$ is tangent to $N$ and so $\xi$ is parallel. 

To prove 2, let $\{X_1,..,X_n\}$ be a $h$-orthonormal frame of $TN$. Then $\{X_1,..,X_n,\xi\}$ is a $h$-orthonormal base of $TM$ and so
$$
\left[ X_1, ...,X_n, D_{X_i}X_j,\xi \right]=\delta_{ij}.
$$
This last equation implies that $\{X_1,...,X_n\}$ is a $g$-orthonormal frame of $TN$ and thus $g=h|_N$.
\end{proof}

\begin{prop}
The vector field $\eta$ belongs to the affine normal plane, is parallel and umbilic.
\end{prop}

\begin{proof}
It is clear that $\eta$ is both parallel and umbilic, and we must verify that $\eta\in\A$. Consider a $g$-orthonormal frame $\{X_1,...,X_n\}$.
Since the metric $g$ is the restriction to $N$ of the Blaschke metric $h$ of $M$, this frame is also $h$-orthonormal. Thus
$$
h(X_1,...,X_n,\xi)=1,
$$ 
which implies that
$$
\left[  X_1,...,X_n,\xi,\eta \right]=1.
$$
This last equation implies that $\eta\in\A$. 
\end{proof}

Fix $p=(x_1,...,x_{n+2})\in N$ and denote by $\xi(p)$ the Darboux vector field at $p$. The tangent space $T_pN$ is orthogonal to both $p$ and $\xi(p)$. In particular,
$T_pN$ is contained in $x_{n+2}=c$, for some $c$, if and only if the vectors $p$, $\xi(p)$ and $e_{n+2}$ are coplanar. 

\begin{prop}
Let $N$ be a submanifold contained in a hyperquadric $M$. Then the immersion $N\subset M$ is umbilic if and only if $N$ is contained in a hyperplane.
\end{prop}

\begin{proof}
Assume first that $N\subset M$ is umbilic and let $L$ be the hyperplane orthogonal to the line $\B$. Assuming that $\B$ is parallel to $e_{n+2}$, the above remark 
implies that $T_pN$ is contained in $x_{n+2}=c$, for some $c$. This implies that $N$ is contained in a hyperplane. 

Conversely, assume that $N=L\cap M$, where $L$ is the hyperplane given by $x_{n+2}=c$. Then we can write $p+t\xi=\lambda e_{n+2}$, for some scalars $t$ and $\lambda$. From
$<\xi,p>=0$, $<p,p>=1$, $<p,e_{n+2}>=c\epsilon_{n+2}$ we obtain
$$
1=\lambda c\epsilon_{n+2},
$$
thus proving that $\lambda$ is independent of $p\in N$. We conclude that such an immersion is umbilic. 
\end{proof}

\section{Umbilic and normally flat immersions}\label{sec:Umbilic}

In this section, we give a geometric characterization of umbilic and normally flat immersions for $n\geq 2$ and 
immersions with $\sigma$ and $\mu$ constant, for $n=1$. By propositions \ref{prop:UNF} and  \ref{prop:SigmaMuConstant}, these conditions
are equivalent to $\phi$ being a visual contour with a constant point $Q$ in the affine normal planes. In order to keep the statements shorter, 
we shall refer to these immersions as umbilic and normally flat even in case $n=1$.

\subsection{Main theorem}

Consider a non-degenerate immersion $f:U\subset\R^{n}\to\R^{n+1}$ and denote by $\nu:U\to\R^{n+1}_{*}$ its Blaschke co-normal map. Fix an origin $O\in\R^{n+1}$ 
and define the immersion $\phi:U\to\R_{*}^{n+1}\times\R$ by 
\begin{equation}\label{eq:UmbilicPhi}
\phi(u)=\left(  \nu(u),  \nu(u)\cdot(f(u)-O) \right), 
\end{equation}
where $u=(u_1,...,u_n)\in U$ and $\nu(u)\cdot(f(u)-O)$ is the {\it affine distance}, or {\it support function}, of $f$ with respect to $O$ (\cite{Cecil}). 
Take $\xi=\phi$ so that $\phi$ becomes a visual contour. 
Our main theorem is the following:

\begin{thm}\label{thm:Umbilic}
Any immersion given by equation \eqref{eq:UmbilicPhi} is umbilic and normally flat. Conversely, any umbilic and normally flat
immersion $\phi:U\to M$ is given by equation \eqref{eq:UmbilicPhi}, for some immersion $f:U\to\R^{n+1}$ and origin $O\in\R^{n+1}$.
\end{thm}

\subsection{Proof of the main theorem: Direct part}

In this section we prove the direct part of theorem \ref{thm:Umbilic},  namely, that the immersion $\phi$ defined by equation \eqref{eq:UmbilicPhi} is umbilic and normally flat.

\begin{proof}
By proposition \ref{prop:UNF}, we have to show that $Q=(0,1)$, $0\in\R^{n+1}$, belongs to the affine normal plane, for any $u\in U$. Take a tangent frame 
$\{X_1,..,X_n\}$ orthonormal with respect to the Blaschke metric $h$ of the immersion $f$. 
Since 
\begin{equation*}
\phi_{*}X=\left(  \nu_{*}X, \nu_{*}X\cdot(f(u)-O) \right),
\end{equation*}
we have that 
\begin{equation*}
D_X\phi_{*}Y=\left(  D_X\nu_{*}Y, D_X\nu_{*}Y\cdot(f(u)-O) \right)- h(X,Y)Q. 
\end{equation*}
The first parcel in the second member is tangent to $M$: In fact, writing
\begin{equation*}
D_X\nu_{*}Y=\sum_{i=1}^n a_i\nu_{*}X_i+b\nu.
\end{equation*}
we obtain
\begin{equation*}
\left( D_X\nu_{*}Y,   D_X\nu_{*}Y\cdot(f-O) \right)=\sum_{i=1}^n a_i \phi_{*}X_i+b\phi,
\end{equation*}
which is tangent to $M$. We conclude that $h^2(X,Y)=h(X,Y)$, and so $\{X_1,...,X_n\}$ is $h^2$-orthonormal for the frame $\{\phi,Q\}$.
To conclude that $Q$ is in the affine normal plane we must verify that 
\begin{equation*}\label{eq:DetPhi}
\left[ \phi_{*}X_1,...\phi_{*}X_n, \phi, Q \right]
\end{equation*}
is constant. But 
\begin{equation*}\label{eq:DetPhi1}
\left[ \phi_{*}X_1,...\phi_{*}X_n, \phi, Q \right]=\left[ \nu_{*}X_1,...\nu_{*}X_n, \nu \right]=-1,
\end{equation*}
thus proving the direct part of theorem \ref{thm:Umbilic}.
\end{proof}

Choose $\xi=\phi$ and $\eta=-Q$ in \eqref{eq:Lap} to get 
\begin{equation}\label{eq:Lap1}
\frac{1}{n}\Delta\phi=-Q-\lambda\phi,
\end{equation}
where $\lambda=-\frac{1}{n}tr_g(h^1)$.

\begin{cor}
We have that $\lambda=H$, where $H$ is the affine mean curvature of the immersion $f$.
\end{cor}
\begin{proof}
From the proof of theorem \ref{thm:Umbilic} we obtain that the Blaschke metric $h$ of $f$ coincides with the metric $g=h^2$ of the immersion $\phi$.
Thus the Laplacian operators of $f$ and $\phi$ are the same. If follows from \cite{Nomizu}, propositions 6.2 and 6.3, that 
$$
\frac{1}{n}\Delta\nu=-H\nu;\ \  \frac{1}{n}\Delta(\nu\cdot(f-O))=-1-H \nu\cdot(f-O).
$$
Thus
$$
\frac{1}{n}\Delta\phi=-H\phi-Q.
$$
Comparing with equation \eqref{eq:Lap1}, the corollary is proved.
\end{proof}

\subsection{Proof of the main theorem: Converse part}

In this section we prove the converse of theorem \ref{thm:Umbilic}, namely, that  any umbilic and normally immersion is given by equation \eqref{eq:UmbilicPhi}, for some immersion $f$
and origin $O\in\R^{n+1}$.

\begin{proof}
Assume that $\phi$ is an umbilic immersion and normally flat and write $\phi=(\psi,z)$. Define $f$ by the conditions
$$
\psi\cdot (f-O)=z;\ \ \psi_{*}X\cdot (f-O)=X(z),
$$
for some origin $O\in\R^{n+1}$. These equations imply that $\psi\cdot  f_{*}X=0$, for any $X$. We conclude that $\psi=\lambda\nu$, for some $\lambda\in\R$, where $\nu$
denotes the Blaschke co-normal of the immersion $f$. 
Since by proposition \ref{prop:UNF} $Q$ belongs to the affine normal plane of $\phi$, we can find a local frame $\{X_1,..,X_n\}$, $h^2$-orthonormal, such that
\begin{equation*}
\left[ \phi_{*}X_1,...\phi_{*}X_n, \phi, Q \right]=-1.
\end{equation*}
This equation implies 
\begin{equation*}
\left[ \psi_{*}X_1,...\psi_{*}X_n, \psi \right]=-1,
\end{equation*}
and then 
\begin{equation}\label{eq:DetNu1}
\left[ \nu_{*}X_1,...\nu_{*}X_n, \nu \right]=-\lambda^{n+1}.
\end{equation}

Differentiating 
\begin{equation*}
\phi_{*}X=\left( \psi_{*}X, \psi_{*}X\cdot(f-O) \right)
\end{equation*}
we obtain
\begin{equation*}
D_X\phi_{*}Y=\left( D_X\psi_{*}Y, D_X\psi_{*}Y\cdot(f-O) \right) + (\psi_{*}Y\cdot f_{*}X)\ Q.
\end{equation*}
The same argument as above says that the first vector of the second member is tangent to $M$, and so $h^2(X,Y)=-\psi_{*}Y\cdot f_{*}X$. We conclude that 
$$
h^2(X,Y)=-\lambda\nu_{*}Y\cdot f_{*}X=\lambda h(X,Y).
$$
Now consider the $h$-orthonormal basis $\{Y_1,..,Y_n\}$, where $Y_i=\lambda^{1/2}X_i$. 
Since 
\begin{equation*}
\left[ \nu_{*}Y_1,...\nu_{*}Y_n, \nu \right]=-1,
\end{equation*}
we conclude that 
\begin{equation}\label{eq:DetNu2}
\left[ \nu_{*}X_1,...\nu_{*}X_n, \nu \right]=-\lambda^{\frac{n}{2}}.
\end{equation}
From equations \eqref{eq:DetNu1} and \eqref{eq:DetNu2} we conclude that $\lambda=1$. 
\end{proof}

\subsection{Umbilic and normally flat immersions contained in a hyperplane}

We now characterize those umbilic and normally flat immersions that are contained in a hyperplane:

\begin{prop}
Assume that $\phi\subset M$ is given by equation \eqref{eq:UmbilicPhi}, for some immersion $f$ and some origin $O$. Then $\phi$ is contained in a hyperplane if and only if 
$f$ is a proper affine sphere.
\end{prop}
\begin{proof}
The affine distance $z$ is constant if and only if $f$ is a proper affine sphere and $O$ its center (\cite{Nomizu}, prop.5.10). Thus, if $f$ is a proper affine sphere and $O$ is arbritary, $\phi$ 
is contained in a hyperplane. Conversely, if $\phi$ is contained in a hyperplane, $z$ is constant for some choice of $O$. Thus $f$ is an affine sphere. 
\end{proof}

Since a compact affine sphere is necessarily an ellipsoid, we have the following corollary:

\begin{cor}
Assume that $\phi\subset M$ is given by equation \eqref{eq:UmbilicPhi}, for some compact immersion $f$ and some origin $O$. Then $\phi$ is contained in a hyperplane if and only if 
$f$ is an ellipsoid.
\end{cor}

\bibliographystyle{amsplain}

\end{document}